\documentclass[10pt]{amsart}
\usepackage{amsmath,amssymb}

\usepackage{amsmath}
\usepackage{amssymb}
\usepackage{bbm}

\usepackage{color}

\newtheorem{theorem}{Theorem}
\newtheorem{corollary}{Corollary}
\newtheorem{lemma}[theorem]{Lemma}

\newtheorem{proposition}[theorem]{Proposition}

\begin{document}

\title{Stabilization of Coefficients for Partition Polynomials }

\author{Robert P.  Boyer}
\address{Department of Mathematics, Drexel University, Philadelphia, PA}

\author{William J. Keith}
\address{CELC, Universidade Lisboa, 1649-003 Lisboa, Portugal}

\email{boyerrp@drexel.edu}

\email{william.keith@gmail.com}

\begin{abstract}  We find that a wide variety of families of partition statistics stabilize in a fashion similar to $p_k(n)$, the number of partitions of $n$ with $k$ parts, which satisfies $p_k(n) = p_{k+1}(n+1)$, $k \geq n/2$.  We bound the regions of stabilization, discuss variants on the phenomenon, and give the limiting sequence in many cases as the coefficients of a single-variable generating function.  Examples include many statistics that have an Euler product form, partitions with prescribed subsums, and plane overpartitions.
\end{abstract}


\thanks{This work was started when William Keith was on the faculty at Drexel University.}

\maketitle

\section{Introduction}

Consider a set of combinatorial objects $\{\alpha\}$ with statistics $wt(\alpha)$ and $t(\alpha)$, thinking of $wt(\alpha)$ as the primary descriptor.  Let $G(z,q)$ be its two-variable generating function, that is, if $p(n,k)$ is the number of objects $\alpha$ with $wt(\alpha) = n$ and $t(\alpha) = k$, 
 \[
 G(z,q) = \sum_{n,k \in \mathbb{N} \bigcup \{0\}} p(n,k) q^n z^k\,  \text{.}
 \]
    
An important example occurs when the generating function
 $G(z,q)$ has the Euler product form 
\begin{equation}\label{eq:product}
G(z,q) = \prod_{i \geq 1} \frac{1}{(1-z q^i)^{a_i}}
\end{equation}
with $a_i \in \mathbb{Z}^+ \bigcup \{0\}$.  
We let  $F_n(z)$ denote the $q^n$ coefficient of $G(z,q)$ which is a polynomial in $z$, so
\[
G(z,q) = \sum_n F_n(z) q^n, \quad  \, F_n(z) = \sum_k p(n,k) z^k \, .
\]

For $a_i = 1$ this is the generating function for partitions of weight $wt(\alpha) = n$ with number of parts $t(\alpha) = k$.  It is a well-known fact in partition theory that $p(n,k)$ for $k$ nearly equal to $n$ has a value independent of $n$: $p(n,n-b)$ is the number of partitions of $b$ for $b\leq\frac{n}{2}$.  A similar result holds for plane partitions of $n$ indexed by their trace \cite[p. 199]{Andrews_book} or \cite[Corollary 5.3]{Stanley}.

This phenomenon, which we call {\sl stabilization,}  is widespread in generating functions of combinatorial interest, even those of greater complexity.  The purpose of this paper is to describe this behavior in more general cases, and consider some illustrative examples and variations.  
We found the polynomial framework to be well suited to these problems rather than a direct approach.  The arguments should be adaptable to a wide variety of cases.

We would like to thank the anonymous referee for a careful readthrough: noting typos, improving exposition, and suggesting occasional strengthening of theorems for which we had been hesitant to extend our reach.  This article is substantially improved from their efforts.

\section{ Basic Infinite Product Generating Functions }

Let  $G(z,q)$ be an infinite product generating function of the form
\begin{equation}\label{eq:new_genfct}
G(z,q) = \sum_{n=1}^\infty F_n(z) q^n =  \prod_{j=1}^{\infty} \frac{1}{(1-z^{b(j)} q^{c(j)})^{a_j}}.
 \end{equation}
 where we assume that the number of $j$ for which $c(j) = t$ for any $t$ is finite, so that the series converges. We find that if the $c$ grow sufficiently faster than the $b$, the upper ends of the $F_n(z)$ stabilize, to the coefficients of a single-variable generating function which we can give.
  
  Let ${\mathcal F}$ denote the set of all nonnegative integer sequences with finite support.
  For ${\bf e} = (e_1,e_2,\dots) \in {\mathcal F}$
set
 \[
 \mu(\mathbf{e}) = (c(1)^{e_1}c(2)^{e_2}\cdots),\quad
 \nu(\mathbf{e}) = (b(1)^{e_1}b(2)^{e_2}\cdots)
 \]
to denote  the partitions with parts $c(j)$ (resp. $b(j)$) appearing $e_j$ times. 
A direct expansion of the generating functions yields
an explicit form  for the polynomial $F_n(z)$:
\begin{eqnarray}\label{eq:take_away}
F_n (z)
 &= &
\sum_{k=0}^n z^k  \,   \sum_{{\mu(\mathbf{e}) \vdash n} \atop {\nu(\mathbf{e}) \vdash k}} \prod_{i \geq 1} 
  \binom{ a_i+e_i-1}{e_i}
  \end{eqnarray}
  We will compare the coefficients of $F_n(z)$ with those of the expansion of
  \[
  \prod_{i=1}^\infty
  \frac{1}{ (1-z^{b(i)})^{a_i}} =
  \sum_{k=0}^\infty z^k \, \sum_{ \nu( {\bf e}) \vdash k} \, \prod_{i\geq 1} \binom{ a_i + e_i -1}{ e_i} .
  \]

\begin{theorem}
\label{thm:stabliztion2}
 Suppose that the exponents  satisfy $a_1 = 1 = b(1) = c(1)$.
 If there exists a positive integer $m \geq 2$ such that 
 \begin{equation}\label{eq:stable2}
 m \cdot b(j) \leq c (j), \quad j \geq 2  ,
 \end{equation}
 then 
 \\
 (a)
 for $k > n/m $, $\left[z^k \right] F_n(z) = \left[z^{k+1} \right] F_{n+1} (z)$.
 \\
 (b) 
  If  $c(j) - b(j) > 0$ for all $j>1$ and the set of $j$ for which $c(j) - b(j)$ takes a given value is finite for any fixed difference, 
  then for $\ell \leq \lfloor n/m \rfloor$, 
  \[
  \left[ z^{n-\ell}\right]F_n(z) = \left[z^{\ell}\right] \prod_{j \geq 2} \frac{1}{(1-z^{c(j)-b(j)})^{a_j}} .
  \]
\end{theorem}
\begin{proof}
(a)
From the explicit form (\ref{eq:take_away}) of the polynomial $F_n(z)$, we can expand it as
\begin{eqnarray}  
F_n (z)
 &= &
\sum_{k=0}^n z^k  \,   \sum_{{\mu(\mathbf{e}) \vdash n} \atop {\nu(\mathbf{e}) \vdash k}} \prod_{i \geq 1} 
  \binom{ a_i+e_i-1}{e_i}
  \nonumber
\\ 
&=& \sum_{k=0}^n \sum_{e_1=0}^k z^{b(1) e_1}
 \sum_{{\mu^-(\mathbf{e}) \vdash n - c(1)e_1} \atop {\nu^-(\mathbf{e}) \vdash k - b(1)e_1}} z^{k-b(1)e_1}
  \prod_{i \geq 2}  \left( {{a_i+e_i-1} \atop {e_i}} \right)  .
\nonumber
\end{eqnarray}

Now if the integer sequence ${\bf e}$ gives a contribution to $[z^{k+1}]F_{n+1}(z)$ and $e_1>0$, then we define 
${\bf e}'$ 
as the integer sequence 
all of whose  terms  agree with ${\bf e}$ except for $j=1$ where we set $e_1'=e_1-1$.
In this way, we obtain all the possible terms contributing to $[z^k]F_n(z)$.
Conversely, any term for $[z^k]F_n(z)$ gives a contribution to $[z^{k+1}]F_{n+1}(z)$ by simply adding
1 to its first component.

The result reduces to showing that any contribution to $[z^{k+1}]F_{n+1}(z)$ indexed by ${\bf e}$
must have $e_1>0$.

We introduce the notation for the modified partitions 
\[
 \mu^{-}(\mathbf{e}) = (c(2)^{e_2} c(3)^{e_3} \cdots),\quad
 \nu^{-}(\mathbf{e}) = (b(2)^{e_2}  b(3)^{e_3} \cdots)  .
 \]
Now assume that  the partition with $\mu({\bf e}) \vdash n$ and $\nu({\bf e}) \vdash k$ gives a contribution to $[z^{k+1}]F_{n+1}(z)$ but
  $e_1 = 0$. So $\mu^-(\mathbf{e}) \vdash n$ and $\nu^-(\mathbf{e}) \vdash k$, but if $c(j) \geq m \cdot b(j)$ for all $j \geq 2$, 
  then $\vert \mu^-({\bf e}) \vert \geq m \vert \nu^-({\bf e}) \vert$.
Hence  if $ k > \frac{n}{m}$, we find that $\vert \mu^{-}({\bf e}) \vert > n$, a contradiction.  Thus all terms in both expansions are the same, and so the coefficients are equal.  This proves part (a).

For part (b), we begin by forming a new partition as follows.
First subtract $b(j)$ from each $c(j)$ and consider the partition
$\lambda(\mathbf{e}) = ((c(2)-b(2))^{e_2}\dots)$.  This removes exactly the amount $\vert \nu^{-}(\mathbf{e})\vert$ from $\vert \mu^{-}(\mathbf{e})\vert$, so
\begin{eqnarray}
F_n (z)
&=& \sum_{k=0}^n \sum_{e_1=0}^k z^{k}
 \sum_{{\lambda(\mathbf{e}) \vdash n - k} \atop {\nu^-(\mathbf{e}) \vdash k - b(1)e_1}} \prod_{i \geq 2}  \left( {{a_i+e_i-1} \atop {e_i}} \right)
\nonumber
\\
&=& \sum_{k=0}^n z^{k}
 \sum_{{\lambda(\mathbf{e}) \vdash n - k} \atop {\vert \nu^-(\mathbf{e}) \vert \leq k}} \prod_{i \geq 2}  \left( {{a_i+e_i-1} \atop {e_i}} \right) . \nonumber
\end{eqnarray}

By hypothesis, the parts of $\lambda$ satisfy $c(j) - b(j) \geq (m-1)b(j)$.  
If $k = n-\ell > \lfloor n/m \rfloor$, then $\ell < \lfloor n \frac{m-1}{m} \rfloor$. 
 If $\lambda(\mathbf{e}) \vdash n-k = \ell$, then $\nu^-(\mathbf{e}) \leq \frac{\ell}{m-1} \leq k$.  
 Thus the sum runs over all $\mathbf{e}$ for which $\lambda(\mathbf{e}) \vdash \ell$ with parts $c(j) - b(j)$.  But this is exactly the coefficient of $z^{\ell}$ in the expansion claimed: 
 \[
 \left[ z^{n-\ell}\right]F_n(z) = \left[z^{\ell}\right] \prod_{j \geq 2} \frac{1}{(1-z^{c(j)-b(j)})^{a_j}} \, \text{ .}
 \]
\end{proof}

By the technique of proof, we have a slight improvement in a special case of hand enumerators of prefabs
 (\cite{Bender_Goodman}, \cite[page 92]{Wilf}).

\begin{corollary}
Suppose $a_1=1$.
If  $F_n(z) = [q^n] \prod_{j=1}^\infty (1-zq^j)^{- a_j}$,
 then
  for $k  \geq n/2$, $\left[z^k \right] F_n(z) = \left[z^{k+1} \right] F_{n+1} (z)$.
\end{corollary}

By the proof, we find that the generating function for the stabilized coefficients gives an upper bound outside the range of stability.

\begin{corollary}
With the hypotheses of the Theorem,
\[
 \left[ z^{n-\ell}\right]F_n(z) \leq \left[z^{\ell}\right] \prod_{j \geq 2} \frac{1}{(1-z^{c(j)-b(j)})^{a_j}} ,
 \quad 0 \leq \ell \leq n. 
 \]
\end{corollary}

Variants of stabilization exist in several guises.  If the $b$ grow faster than the bound of the previous theorem, we find that the smaller end of the polynomials stabilize instead of the larger (and $b$ and $c$ staying within a given ratio range will permit both phenomena).  We also note that it is possible for the larger coefficients of a sequence of polynomials to stabilize in periods, i.e., the coefficients match those of a polynomial every 2 or more steps further along.

\begin{theorem}\label{thm:stabilization2}
(a)
Let $\{ b(j)\}$ and $\{ c(j)\}$ be two strictly increasing sequences of integers, positive except that 
$b(1)=0$. 
 Let $F_n(z) = [q^n]\prod_{j=1}^\infty (1-z^{ b(j) } q^{ c(j) })^{  - a_j  }$ with $a_1=1$.
If  there exists a positive integer $m $ such that for all $j \geq 2$
 \[
 (m+1) b(j) > c(j) ,
 \]
then 
\[
[z^k] F_n(z)  = \, [ z^k] F_{n+ c(1)}(z), \quad 0 \leq k \leq n/(m+1).
\]
(b)
Let $a_1=0, a_2=1$. 
If
$F_n(z) = [q^n]\prod_{j=2}^\infty (1-z q^{ j})^{  - a_j  }$, then
$\deg(F_n) = \lfloor n/2 \rfloor$
and
\[
[z^k] F_n(z) = \, [z^{k+1}] F_{n+2}(z) , \quad  k \geq n/3 .
\]
\end{theorem}
\begin{proof}
Let ${\bf e}$ be a sequence of non-negative integers with finite support.
By (\ref{eq:take_away}), for the coefficients $[z^k]F_n(z)$ and $[z^k] F_{n+c(1)}(z)$, we need to consider the following two sets of 
finitely supported nonnegative integer sequences:
\begin{eqnarray*}
S_n &=&
 \{ {\bf e} \in {\mathcal F} : | \mu({\bf e}) | = n,  \, | \nu({\bf e} ) | =k \}, 
 \\
S_{n+c(1)} &=&
  \{ {\bf f} \in {\mathcal F} : | \mu({\bf f}) | = n+c(1),  \, | \nu({\bf f} ) | =k \} .
\end{eqnarray*}
We construct a bijection between these two sets. Given ${\bf e} \in S_n$, we take the
corresponding ${\bf f}$ with $f_i=e_i$ for $2 \leq i$ since $b(1)=0$.
Next write out $ | \mu({\bf e}) | = n$
 and $ | \mu({\bf f}) | = n+c(1)$:
\[
n= \sum_{j \geq 1} c(j) e_j, \quad n+c(1) = c(1) f_1 + \sum_{j \geq 2} c(j) e_j .
\]
The term $f_1$ uniquely determines a preimage $e_1$ 
provided $f_1>0$. But $f_1$ must be positive for $k \leq n/(m+1)$ since
\[
\sum_{j\geq 2} c(j) f_j = \sum_{j \geq 2} c(j) e_j < (m+1) \sum_{j \geq 2} b(j) e_j  = (m+1)k  \leq n.
\]
Finally, the coefficients themselves agree; that is,  $[z^k] F_n(z) = [z^k] F_{n+c(1)}(z)$:
\[
  \prod \left\{ \binom{ a_i + e_i -1}{ e_i} : {\bf e} \in S_n \right\} =
  \prod \left\{ \binom{ a_i + f_i -1}{ f_i} : {\bf f} \in S_{n+c(1)} \right\} 
\]
since the above binomial coefficients are all equal for $i \geq 2$ and when $i=1$ they both reduce to $1$ since $a_1=1$.

For part (b), let
$S_n = \{ {\bf e } \in {\mathcal F} :
\mu({\bf e}) \vdash n, \nu({\bf e} ) \vdash k\}$
while
$S_{n+2}=
 \{ {\bf f } \in {\mathcal F} :
\mu({\bf f}) \vdash n+2, \nu({\bf f} ) \vdash k+1\}$.
We can construct a bijection between these two sets as in part (a) provided $f_2 >0$ when $k \geq n/3$.
Assume that $f_2=0$ is possible. Then $\sum f_j = k+1$ while $\sum_{j \geq 3} j f_j = n+2$. On the other hand,
$3(k+1) \geq \sum_{j \geq 3} j f_j $ which yields a contradiction.
\end{proof}

 \section{Partitions With Prescribed Subsums}\label{section:subsums}

Fix a positive integer $m$  and integer $i$ so $1 \leq i \leq m$.  Canfield-Savage-Wilf \cite[Section 3]{CSW} introduced the generating function
\[
G_{m,i}(z,q) = \prod_{j=1}^\infty \, 
\prod_{b=1}^{i-1} \frac{1}{ 1- z^{j-1} q^{ (j-1)m+b}} \, \prod_{b=i}^m \frac{1}{ 1- z^{j} q^{ (j-1)m+b}}
\]
to describe partitions with prescribed subsums. They let
$\Lambda_{m,i}(n,k)$ be the number of partitions $\lambda=(\lambda_1,\cdots, \lambda_n)$ of $n$
such that the sum of those parts $\lambda_j$ whose indices $j$ are congruent to $i$ modulo $m$ is $k$;
that is, 
\[
\sum_{ j :  j \equiv i \, ({\rm mod} \, m)} \lambda_j = k .
\]
Then they found
\[
\sum_{n,k \geq 0} \Lambda_{m,i}(n,k) z^k q^n = G_{m,i}(z,q).
\]

We begin by recovering a result for 
$\Lambda_{2,2}(n,k)$ in \cite[Theorem 1]{CSW} and \cite{Yuri}
by reformulating it  in terms of the generating function
$G_{2,2}(z,q)$  and stabilization of polynomial coefficients.

\begin{proposition}\label{prop:n-mk}
Let $m \geq 2$ and $1 \leq b < m$. Let $G(z,q)= \prod_{j=1}^\infty (1-z^{j-1} q^{(j-1)m+b} )^{-1}$
and $A_n(z) = [q^n] G(z,q)$.  Then for $ 0 \leq k \leq n/(m+1)$  we have
\begin{enumerate}
\item\quad
$[z^k] A_n(z) = [z^k] A_{n+b}(z)$,
\item\quad
if $n-mk$ is not divisible by $b$, then
$[z^k] A_n(z) = 0$,
\item\quad
if $n-mk$ is divisible by $b$ and $bk \leq n/(m+1)$, then
$[z^k] A_n(z) = p(k)$.
\end{enumerate}
\end{proposition}
\begin{proof}
The first part is a direct consequence of the last theorem.

For part (2), in Theorem \ref{thm:stabilization2} let $c(j) = m(j-1)+b$ and $b(j)=j-1$.  When all $a_j =1$, $[z^j]F_n(z)$ is the number of all ${\bf e} \in {\mathcal F}$ such that $\mu({\bf e}) \vdash n$ and $\nu({\bf e}) \vdash k$.
Consider
\begin{eqnarray}\label{eq:n-mk}
n
&=&| \mu({\bf e})| = \sum_{j \geq 1} e_j [ m(j-1) + b] 
\nonumber
\\
&=&
m \sum_{j \geq 1} e_j (j-1) + b  \sum_{j \geq 1} e_j 
\nonumber
\\
&=&
mk + b  \sum_{j \geq 1} e_j .
\end{eqnarray}
Hence $n-mk$ must be divisible by $b$ for any nonzero choice of {\bf e}.

For part (3),
assume $0 \leq bk \leq n/(m+1)$ and that $n-mk$ is divisible by $b$. 
Let {\bf e} be any solution to $\nu( {\bf e})=k$. Note that this does not give
a constraint for the choice of $e_1$. On the other hand, the choice of $e_1$ by (\ref{eq:n-mk}) must be
\[
e_1 = (n-mk)/b - \sum_{j\geq 2} e_j, \quad e_1 \geq 0
\]
to yield $\mu ({\bf e}) = n$. 
Hence, $0\leq [z^k] F_n(z) \leq p(k)$.
Finally,
the inequality $0 \leq bk \leq n/(m+1)$ shows that $e_1\geq 0$ always holds. Hence, 
$[z^k] F_n(z)=p(k)$.
\end{proof}

 \begin{theorem}
  Let  $G_A(z,q)=\prod_{k=1}^\infty ( 1-z^{k} q^{2k-1} )^{-1}$
  and $G_B(z,q)= \prod_{k=1}^\infty ( 1-z^{k} q^{2k} )^{-1}$, so $G_{2,2}(z,q) = G_A(z,q) G_B(z,q)$. Then the coefficients of the polynomials
 $F_n(z)= [q^n]  G_{2,2}(z,q)$ satisfy
  \[
   [z^{k}] F_n(z) = [z^{k}] F_{n+1}(z) = \sum_{\ell=0}^k p(\ell) p(k-\ell), \quad 0 \leq k \leq n/3,
  \]
  where $p(\ell)$ is the number of partitions of $\ell$, as usual.
  \end{theorem}
  \begin{proof} 
 Let $A_n(z)=[q^n]G_A(z,q)$ and $B_n(z) = [q^n]G_B(z,q)$. The generating function $G_B(z,q)$ has the explicit expansion
  \[
  G_B(z,q) = \sum_{j=0}^\infty p(j) z^j q^{2j} 
  \]
  so $B_{2k+1}(z)=0$ while $B_{2k}(z) = p(k) z^k$.
  We also know by Proposition \ref{prop:n-mk} that
  \[
  [z^j] A_n(z) = p(j), \quad 0 \leq j \leq n/3.
  \]
  Next we have that
  \[
  F_n(z) = \sum_{\ell=0}^{n/2} A_{n-2\ell}(z) B_{2\ell}(z) = \sum_{\ell=0}^{n/2}  p(\ell) z^\ell A_{n-2\ell}(z)
  \]
  Examine the coefficient $[z^k]F_n(z)$ for $0 \leq k \leq n/3$:
  \[
  [z^k]F_n(z) = [z^k] \sum_{\ell=0}^{n/2}  p(\ell) z^\ell A_{n-2\ell}(z)
  =\sum_{\ell=0}^{n/2} p(\ell) \, [z^{k-\ell} ] A_{n-2\ell}(z)
  \]
  Since $k-\ell \leq (n-2\ell)/3$ for $0\leq \ell \leq k$ and $k\leq n/3$,
  we find    $[z^{k - \ell}] A_{n - 2 \ell}(z) = p(k - \ell)$.
  We conclude that
    \[
  [z^k]F_n(z)= \sum_{\ell=0}^{k} p(\ell) \, p( k - \ell) .
  \]
   \end{proof}

In order to investigate more general cases for the
 polynomials $F_n(z) = [q^n] G_{m,i}(z,q)$ it is convenient to rewrite the generating function as
\begin{equation}\label{eq:new_eq}
G_{m,i}(z,q)=
\left( \prod_{b=1}^{i-1} \frac{1}{1-q^b} \, \right) \,
\prod_{a=1}^\infty \, \prod_{d=0}^{m-1} \frac{1}{ 1-z^a q^{ i+ (a-1) m + d}}.
\end{equation}
We wish to find a useful form for the coefficient $[z^j] F_n(z)$.  As usual, we have
\begin{eqnarray*}
[q^n] G_{m,i}(z,q)
&& = 
\sum_{s=0}^n [q^{n-s}] \left( \prod_{b=1}^{i-1} \frac{1}{1-q^b} \, \right) \, 
[q^s] \prod_{a=1}^{\infty} \prod_{d=0}^{m-1} \frac{1}{ 1- z^{a} q^{i+ (a-1)m+d}}
\\
&&= \quad
\sum_{s=0}^n [q^{n-s}] \left( \prod_{b=1}^{i-1}\, 
 \sum_{ f_b=0}^\infty \,  \left(q^b \right)^{f_b} \, \right) 
\\
&&
\qquad
\times \quad
[q^s] \prod_{a=1}^{\infty} \prod_{d=0}^{m-1} \, \sum_{ f_{(a,d)}=0}^\infty 
\left(z^{a} q^{ i+(a-1)m+d}  \right)^{f_{(a,d)}} 
\end{eqnarray*}
A typical term of the above expansion is indexed by a pair of partitions $(\rho, \mu)$
where $\rho \vdash n-s$ where each part of $\rho$ is $<i$ while $\mu \vdash s$ where
each part of $\mu$ is $\geq i$. We write the parts of $\mu$ as $\mu_{(a,d)}$ with multiplicity
$e_{\mu(a,d)}$ where $a \geq 1$ and $0 \leq d < m$. By construction, we find that
\[
s = \sum_{a,d} e_{\mu(a,d)} \mu_{(a,d)}.
\]
while, of course, no part of $\rho$ has a contribution to $z$.

Each part $\mu_{(a,d)}$ of $\mu$ with multiplicity $e_{\mu(a,d)}$ contributes the power of $z$
\[
 \left( z^a \right)^{e_{\mu(a,d)}} = z^{ a e_{ \mu(a,d)}},
\]
so the total contribution from $\mu$ is
\[
\sum_{a,d} a e_{\mu(a,d)}.
\]

We prove a few lemmas in preparation for the next theorem.

\begin{lemma}
\begin{eqnarray}\label{eq:coeff_F}
\lefteqn{
[z^j] F_n(z)
}
\nonumber
\\
&& = \sum_{s=ji}^n  p(n-s, < i) \,  \#\{  \mu \vdash s :   \textrm{all parts of } \mu \textrm{ are } \geq  i, \sum_{a,d} a e_{\mu(a,d)}=j   \}
\end{eqnarray}
where $p(n-s, < i)$ denotes the number of all partitions of $n-s$ all of whose parts are strictly less than $i$.
\end{lemma}
\begin{proof}
This expression for $[z^j] F_n(z)$ follows directly from the generating function except for the
range of summation.
Let $\mu$ be a partition of $s$. Then we have
\begin{eqnarray*}
s &=& \sum_{a,d} [ i + m (a-1) + d] e_{\mu(a,d)}
\\
&\geq& \sum_{a,d} [ i + i (a-1) ] e_{\mu(a,d)}
\\
&=&
i \sum_{a,d} a e_{\mu(a,d)} = ij .
\end{eqnarray*}
\end{proof}

Let $r$ denote the number of  parts of $\mu$ while $t$ denotes  the number of parts with $a\geq 2$ so $r-t$ gives the number of 
parts with $a=1$. Note that
\[
r = \sum_{a,d} e_{\mu(a,d)}.
\]

\begin{lemma}  Given the partition $\mu$, we have the bound $j-r \geq t$, and can rewrite $s$ as 
$$s = i r + m (j-r) + \sum_{a,d} d e_{\mu(a,d)}.$$
\end{lemma}
\begin{proof}
Since $\mu \vdash s$, we find that
\begin{eqnarray*}
s &=&
\sum_{a,d} e_{\mu(a,d)} \, [ i + (a-1) m + d ]
\\
&=&
i \, \sum_{a,d} e_{\mu(a,d)}  + m \sum_{a,d} (a-1) e_{\mu(a,d)} + \sum_{a,d} d e_{\mu(a,d)}
\\
&=&
i r + m (j-r) + \sum_{a,d} d e_{\mu(a,d)}.
\end{eqnarray*}
For the bound, consider
\[
j-r = \sum_{a,d} (a-1) e_{\mu(a,d)} \geq \sum_{a \geq 2, d} e_{\mu(a,d)} \geq t.
\]

\end{proof}

\begin{lemma}
Let $\mu$ be a partition of $s$ all of whose parts are $\geq i$. If
$m \geq i+2$ and
$j > s /(i+1)$, then $\mu$ must have at least one part of size $i$.
\end{lemma}
\begin{proof}
We now assume that $m \geq i+2$ and that the part $i$ does not appear in $\mu$.
In particular,  for any part of $\mu$ of the form $i +d$, with $a=1$, $d$ must be strictly positive.
Then we have a refinement of the above bounds:
\begin{eqnarray*}
s &\geq&
i r + m (j-r) + \sum_{a,d} d e_{\mu(a,d)}
\geq
i r + (i+2) (j-r) + \sum_{ d} d e_{1,d}
\\
&\geq&
i r + (i+2) (j - r) + (r-t)
\\
&\geq&
 i r + (i+1) (j - r) + t + (r-t)
 \\
 &=&
 (i+1) [ r + (j-r)] = (i+1) j.
\end{eqnarray*}
Hence the partition $\mu$ must have $i$ as a part; otherwise, $(i+1)j >s$  which contradicts our assumption.
\end{proof}

With these preparations, we will show that

\begin{theorem}\label{thm:subsums}
Let $m \geq 1$ and $1 \leq i \leq m$. Set $F_n(z) = [ q^n] G_{m,i}(z,q)$ where
$G_{m,i}(z,q)$ is given by (\ref{eq:new_eq}).
 If $m>i+1$  and $j> \frac{n}{i+1}$, then 
\[
[z^j] F_n(z) = [z^{j-1}] F_{n-i}(z)  .
\]
\end{theorem}
\begin{proof}
By (\ref{eq:coeff_F}), we need to show that
\begin{eqnarray*}
&&
 \sum_{s=ji }^{ n}  p(n-s, < i) \, \# \{   \mu \vdash s :  \textrm{all parts of } \mu \textrm{ are } \geq  i, \sum_{a,d} a e_{\mu(a,d)}=j   \}
 \\
 &&
 =
  \sum_{s'= (j-1) i }^{n- i }  p(n-s'-i, < i) \,
  \\
  &&
  \quad
  \times  \quad
  \# \{    \nu \vdash s'  :  \textrm{all parts of } \nu \textrm{ are } \geq  i, \sum_{a,d} a e_{\nu(a,d)}=j  -1 \}
\end{eqnarray*}
These two coefficients are equal provided we construct a bijection $T$ between
\begin{eqnarray*}
U_s &=&  \{   \mu \vdash s :  \textrm{all parts of } \mu \textrm{ are } \geq  i, \sum_{a,d} a e_{\mu(a,d)}=j   \}
 \\
 V_{s'}
 &=&
  \{    \nu \vdash s'  :  \textrm{all parts of } \nu \textrm{ are } \geq  i, \sum_{a,d} a e_{\nu(a,d)}=j  -1 \}.
\end{eqnarray*}
when $s' = s-i$. Let $\mu \in U_s$. By Lemma 7, the partition $\mu$ must have a part equal to $i$. 
Let $\nu=T(\mu)$ be the partition of $s-i$ obtained by deleting one part from $\mu$ of size $i$. It is easy
to verify that $\nu \in V_{s'}$. The inverse of $T$ is simply adding a part of size $i$ to $\nu \in V_{s'}$.
 \end{proof}
 
In other words, Theorem \ref{thm:subsums} shows that if $m>i+1$ and $j> n/(i+1)$, then the subsums satisfy
$\Lambda_{m,i}(n,j)= \Lambda_{m,i}(n-i,j-1)$.

\section{Laurent Type Polynomials }\label{section:laurent}

A more general case consists of generating functions that involve $z$ raised to different powers; 
ultimately, we might treat the case of the generating function 
\[
G(z,q) = \prod_{(i,j) \in \mathbb{Z}^2 \setminus (0,0)} {(1 - z^i q^j)}^{a_{ij}} \, .
\]

An important example comes from the generating function for the crank statistic for partitions. 
Let
\[
C(z,q) = \prod_{k\geq 1} \frac{1-q^k}{(1-z q^k)(1-z^{-1} q^k)} = \sum_{n=0}^\infty \, M_n(z) q^n 
\]
where $M_n(z)$ is a symmetric Laurent polynomial.
From the definition of the crank, the coefficient of $z^{n-k}$ in $M_n(z)$,
for $ k \leq \frac{n}{2}$, equals  the number of partitions of $k$ that include no 1s.
In particular, the coefficients of $M_n(z)$ stabilize in the ranges for powers $n-k$ and $-n+k$
for $0\leq k \leq \lfloor n/2 \rfloor$.  It is suggested in \cite{BG2} that the zeros  for the crank polynomial converge to
the unit circle. 
Another example is  the generating
function
\[
\prod_{k=1}^\infty \frac{ 1}{ (1-z q^k)^{k}} \, \frac{1}{ (1-z^{-1} q^k)^k } \, \frac{1}{ (1-q^k)^{2k}}.
\]
which comes from the Donaldson-Thomas Theory in algebraic geometry and whose asymptotics were studied in   \cite{siam10}.

\begin{lemma}\label{thm:freezing1}
 Let $\{A_n(z)\}_{n=0}^\infty$ be a sequence of polynomials, such that the degree of $A_n(z)$ is $n$, whose coefficients satisfy
 \[
 [z^{n-k}] A_n(z) = [ z^{n+1-k}]A_{n+1}(z), \quad 0 \leq k \leq n/m
 \]
 for some integer $m \geq 2$.
 Let $\{B_n(z)\}_{n=0}^\infty$ be another sequence of polynomials. Then the coefficients of the 
 polynomial sequence $\{F_n(z)\}_{n=0}^\infty$ 
 where
 \[
 F_n(z) = \sum_{\ell=0}^n \, A_{\ell}(z) B_{n-\ell}(z^{-1})
 \]
 also satisfy
 \[
 [z^{n-k}] F_n(z) = [ z^{n+1-k}]F_{n+1}(z),  \quad 0 \leq k \leq n/m.
 \]
  \end{lemma}
 \begin{proof}
 Let $0 \leq k \leq \lfloor n/m \rfloor$. Consider $[z^{n-k}] F_n(z)$. We have
 \begin{eqnarray}
 [z^{n-k}] F_n(z)
 &=&
 [z^{n-k}]  \sum_{\ell=0}^n A_{\ell}(z) \, B_{n-\ell}(z^{-1})
 \nonumber
 \\
 &=&
\sum_{\ell=0}^n \, \sum_{a=0}^k \, [ z^{ n-k+a}] A_\ell(z) \, [ z^{-a}] B_{n-\ell}(z^{-1})
  \nonumber
 \\
 &=&
\sum_{\ell=n-k}^n \, \sum_{a=0}^{k} \, [ z^{ n-k+a}] A_\ell(z) \, [ z^{-a}] B_{n-\ell}(z^{-1})
 \label{eq:first_sum}
 \end{eqnarray}
where in the last step we note that $[z^{ n-k+a}]A_\ell(z)=0$ if $\ell < n-k$.

For $[z^{n+1-k}]F_{n+1}(z)$, we reindex $\ell$ by 1 and note that $[z^{n+1-k}]A_0(z)=0$ for all $n$, giving us the expression for $[z^{n+1-k}]F_{n+1}(z)$:
 \begin{equation}
   \label{eq:second_sum}
\sum_{\ell=n-k}^{n} \, \sum_{a=0}^{k} \, [ z^{ n+1-k+a}] A_{\ell+1}(z) \, [ z^{-a}] B_{n-\ell}(z^{-1})
\end{equation}
By assumption, we know that
\[
        [ z^{ n-k+a}] A_\ell(z)  =[ z^{ n+1-k+a}] A_{\ell+1}(z), \quad n-k \leq \ell \leq n
\]
since $0 \leq k \leq \lfloor n/m \rfloor$.  The coefficients of $[z^{-a}] B_{n-\ell}(z^{-1})$ are the same, and consequently the two sums (\ref{eq:first_sum})  and (\ref{eq:second_sum}) are
equal.
 \end{proof}

\begin{lemma}\label{lemma:convolution}
Given the generating function
$G(z,q) = \prod_{i \geq 1} (1-zq^i)^{-a_i}$ where $a_1=1$, let $A_n(z) = [q^n] G(z,q)$.  Let $Q(q) = \sum_{j=0}^\infty c_j q^j$.  
Let $F_n(z) = [q^n] G(z,q)Q(q)$. Then the tail coefficients of $F_n(z)$ stabilize; that is,
\[
[z^{k}] F_n(z) = [z^{k+1}] F_{n+1}(z), \quad k \geq n/2.
\]
\end{lemma}
\begin{proof}
By construction, the polynomials $F_n(z)$ have the form
\[
F_n(z) = \sum_{\ell=0}^n c_{n-\ell} A_\ell(z).
\]
Let $k \geq n/2$. Then the coefficients for $[z^{k}] F_n(z)$ and $ [z^{k+1}] F_{n+1}(z)$ are given by
\begin{eqnarray*}
[z^{k}]  F_n(z) = \sum_{\ell=k}^n c_{n-\ell} [z^{k}  ]  A_\ell(z),
\quad
\,[z^{k+1}]F_{n+1}(z) = \sum_{j = k+1}^{n+1} c_{n+1-j} [z^{k+1}] A_{j}(z).
\end{eqnarray*}
As a consequence of Corollary 1, $c_{n-\ell} [z^{k}  ]  A_\ell(z) = c_{n+1-j} [z^{k+1}] A_{j}(z)$ for $j= \ell+1$ and $k \leq \ell \leq n$.
\end{proof}
 
 \begin{theorem}
 If $a_1 = b_1 = 1$, and 
\[
G(z,q) = \sum_{n=0}^\infty F_n (z) q^n = \prod_{i \geq 1} {(1 - z q^i)}^{-a_i} {(1 - z^{-1} q^i)}^{-b_i} {(1 \pm q^i)}^{c_i},
\]
then the coefficients of the Laurent polynomials $F_n(z)$ satisfy
\[
[z^{n-k}] F_n(z) = [z^{n+1-k}] F_{n+1}(z), \quad [z^{-(n-k)}] F_n(z) = [z^{-(n+1-k)}] F_{n+1}(z),
\]
for $0 \leq k \leq n/2$.
 \end{theorem}

\section{ Plane Overpartition  Stabilization }

A plane partition is an array of positive integers, conventionally justified to the upper left corner of the fourth quadrant, which are weakly descending left in rows and down in columns.  A plane overpartition is a plane partition whose entries may be overlined or not according to
certain rules \cite{CSV}: in each row, the last occurrence of an integer may be overlined
(or not) and in every column, all but the first occurrence of an integer are overlined, while the first occurrence may or may not be overlined.
In \cite[Proposition 4]{CSV}, the generating function for the weighted plane overpartitions is found to be 
\[
\sum_{ \Pi \,\,\textrm{is a plane overpartition}}
z^{ o(\Pi)} q^{ | \Pi |} =
\prod_{n=1}^\infty  \frac{ (1+ zq^n)^n}{  (1-q^n)^{ \lceil n/2 \rceil} (1-z^2q^n)^{ \lfloor n/2 \rfloor } }
\]
where $o( \Pi)$ is the number of overlined parts of the plane
overpartition $\Pi$.

\begin{theorem}
Let $G(z,q)$ be the generating function for the polynomials $F_n(z)$:
\[
G(z,q)=
\prod_{n=1}^\infty  \frac{ (1+ zq^n)^n}{  (1-q^n)^{ \lceil n/2 \rceil} (1-z^2q^n)^{ \lfloor n/2 \rfloor } }
=
\sum_{n=0}^\infty F_n(z) q^n .
\]
Then the coefficients of the polynomials $F_n(z)$  satisfy the stabilization condition
\[
[z^{k+1}] F_{n+1}(z) = [ z^k] F_n(z),
\]
for $k\geq  2n/3$.
\end{theorem}
\begin{proof}
Let $\{A_n(z)\}$ be the polynomial sequence with generating function $G_A(z,q)$ where
\[
G_A(z,q) = \prod_{n=2}^\infty \frac{ 1}{ (1-z^2q^n)^{ \lfloor n/2 \rfloor } } =  \sum_{n=0}^\infty A_n(z) q^n
\]
and $\{ B_n(z)\}$ with generating function $G_B(z,q)$:
\[
G_B(z,q) = \prod_{n=2}^\infty (1+ zq^n)^n  = \sum_{n=0}^\infty B_n(z) q^n  .
\]
Easily, we have that $\deg(A_n) = 2 \lfloor n/2 \rfloor$.
By Theorem \ref{thm:stabilization2}, replacing $z$ by $z^2$, we also find
\[
[z^k] A_n = [ z^{k+2}] A_{n+2}, \quad k \geq 2n/3 .
\]

The degree of the polynomial $B_N(z)$ is the largest number of parts in a possible partition of $N$ drawn from a multiset of two 2s, three 3s, etc.  For $N \geq 21 = 2+2+3+3+3+4+4$, the average size of part is at least 3 and so 
\[
\deg(B_N) \leq \frac{1}{3}N \, \text{.}
\]
  For smaller $N$, direct calculation shows that $\deg(B_N) \leq \frac{2}{3}N$.

It is more convenient to work with the intermediate polynomials
\[
Q_n(z) = [q^n] G_A(z,q) G_B(z,q) =\sum_{\ell=0}^n A_{n-\ell}(z) \, B_\ell(z).
\]

We have the summation formula for $[z^k] Q_n(z)$:
\begin{eqnarray*}
[z^k] Q_n(z)
&=&
\sum_{\ell=0}^n \, \sum_{a=0}^k [ z^{k-a}] A_{n - \ell}(z) \cdot [z^a] B_\ell(z)
\\
\\
&=&
\sum_{\ell \in  I_{k,n} } \,   \sum_{ a =0 }^{ \deg(B_\ell)} \, [ z^{k-a}] A_{n - \ell}(z)  \cdot [z^a] B_\ell(z) 
\end{eqnarray*}
where $I_{k,n}$ is the set of indices 
\[
I_{k,n} =
\left\{
\ell :
\deg(B_\ell) + \deg(A_{n - \ell }) \geq k, \,  0 \leq \ell \leq n
\right\} .
\]
We have $I_{k+2, n+2} = I_{k,n} $ since for any $B_\ell$ the matching $A_{n-\ell}$ for sums of $k$ map to the matching $A_{n+2-\ell}$ for sums of $k+2$.

We next show that
\[
\sum_{\ell \in  I_{k,n} } \,   \sum_{ a = 0}^{ \deg(B_\ell)} \, [ z^{k-a}] A_{n - \ell}(z)  \cdot [z^a] B_\ell(z) 
=
\sum_{\ell \in  I_{k+2,n+2} } \,  \sum_{ a = 0}^{ \deg(B_\ell)} \, [ z^{k+2-a}] A_{n+2 - \ell}(z)  \cdot [z^a] B_\ell(z) ;
\]
To do this, we need that  all the terms $[ z^{k-a}] A_{n- \ell}$ fall into the stable range of indices. 
Now
to get into the stable range, we need
\[
k-a \geq \frac{2}{3} ( n  - \ell) 
\]
where $0\leq a \leq \deg B_\ell$ and $\ell \in I_{k,n}$.
We consider  the stronger condition
\begin{eqnarray*}
\frac{2}{3} n - \deg B_\ell  &\geq&  \frac{2}{3} ( n  - \ell) ,
\end{eqnarray*}
which reduces to 
\[
\deg(B_\ell) \leq \frac{2}{3} \ell, \quad \ell \in I_{k,n} 
\]
which does indeed hold.

Our next step is to define
\[
P_n(z) = (1+zq) Q_n(z).
\]
(Leaving out this factor earlier simplified the degree analysis, since without it we had the single exceptional case of $\deg (B_1) = 1$.)  We observe that
\begin{eqnarray*}
\,
[z^k] P_n (z) &=& [ z^k] Q_n(z) + [ z^{k-1}  ] Q_{n-1}(z),
\\
\,
[z^{k+1}]  P_{n+1} (z) &=&  [ z^{k+1}] Q_{n+1}(z) + [ z^{k}] Q_{n}(z) .
\end{eqnarray*}
Hence, we have
\[
[z^k] P_n(z) = [z^{k+1}] P_{n+1}(z), \quad k \geq 2n/3 .
\]

To finish the proof, we define a polynomial family $C_n(q)$ by
\[
\prod_{n=1}^\infty  \frac{ 1 }{  (1-q^n)^{ \lceil n/2 \rceil} } = \sum_{n=0}^\infty C_n(q) .
\]
Then the polynomials $F_n(z)$ are given by
\[
F_n(z) = [q^n] C(q) \, \sum_{\ell =0}^\infty  P_\ell (z) q^\ell .
\]
By Lemma \ref{lemma:convolution}, 
we see that this construction maintains the stability of the coefficients of the polynomial family $\{ P_\ell(z)\}$.
\end{proof}

\begin{corollary}
Let $\overline{pp}_k(n)$ be the number of plane overpartitions of $n$ with $k$ overlined parts. If $k \geq 2n/3$, then
\[
\overline{pp}_k(n) =\overline{pp}_{k+1}(n+1).
\]
\end{corollary}

\end{document}